\documentclass{amsart}
\usepackage{amsmath,amscd,amsthm,amsfonts,amssymb}

\theoremstyle{plain}
\newtheorem{theorem}                 {Theorem}      [section]
\newtheorem{proposition}  [theorem]  {Proposition}

\newtheorem{lemma}        [theorem]  {Lemma}

\theoremstyle{definition}
\newtheorem{example}      [theorem]  {Example}
\newtheorem{remark}       [theorem]  {Remark}
\newtheorem{definition}   [theorem]  {Definition}

\numberwithin{equation}{section}

\DeclareMathOperator{\ad}{ad}
\DeclareMathOperator{\Ad}{Ad}

\def \theo-intro#1#2 {\vskip .25cm\noindent{\bf Theorem #1\ }{\it #2}}

\def \rn{\mathbb R}
\def \cn{\mathbb C}

\def \F{\mathcal F}
\def \H{\mathcal H}

\def \V{\mathcal V}

\def \ip #1#2{\langle #1,#2 \rangle}

\def \lb#1#2{[#1,#2]}

\def \a{\mathfrak{a}}

\def \g{\mathfrak{g}}

\def \k{\mathfrak{k}}
\def \m{\mathfrak{m}}
\def \n{\mathfrak{n}}

\def \s{\mathfrak{s}}

\def \so#1{\mathfrak{so}(#1)}

\def \nab#1#2{\hbox{$\nabla$\kern -.3em\lower 1.0 ex
    \hbox{$#1$}\kern -.1 em {$#2$}}}

\begin{document}
\baselineskip 22pt \larger

\allowdisplaybreaks

\title{Harmonic morphisms from \\homogeneous Hadamard manifolds}

\author{Sigmundur Gudmundsson}
\author{Jonas Nordstr\" om}



\address
{Department of Mathematics, Faculty of Science, Lund University,
Box 118, S-221 00 Lund, Sweden}
\email{Sigmundur.Gudmundsson@math.lu.se}
\address
{Department of Mathematics, Faculty of Science, Lund University,
Box 118, S-221 00 Lund, Sweden}
\email{JonasCWNordstrom@gmail.com}

\begin{abstract}
We present a new method for manufacturing complex-valued harmonic
morphisms from a wide class of Riemannian Lie groups. This yields
new solutions from an important family of homogeneous Hadamard
manifolds.  We also give a new method for constructing left-invariant
foliations on a large class of Lie groups producing harmonic morphisms.
\end{abstract}

\maketitle


\section{Introduction}

The notion of a minimal submanifold of a given ambient space is of
great importance in differential geometry. Harmonic morphisms
$\phi:(M,g)\to(N,h)$ between Riemannian manifolds are useful tools
for the construction of such objects. They are solutions to
over-determined non-linear systems of partial differential
equations determined by the geometric data of the manifolds
involved. For this reason harmonic morphisms are difficult to find
and have no general existence theory, not even locally.  On the
contrary there exist $3$-dimensional Lie groups not admitting any
solutions, independent of which left-invariant Riemannian metrics
they are equipped with, see \cite{Gud-Sve-4} and \cite{Gud-Sve-5}.
This makes the existence theory particularly interesting.

If the codomain is a surface the problem is invariant under
conformal changes of the metric on $N^2$.
Therefore, at least for local studies, the codomain can be
taken to be the complex plane with its standard flat metric.
Complex-valued harmonic morphism $\phi:(M,g)\to\cn$ from Riemannian
manifolds generalize holomorphic functions $f:(M,g,J)\to\cn$
from K\" ahler manifolds and posess many of their highly interesting
properties. The theory of harmonic morphisms can therefore be
seen as a generalization of complex analysis.

In this paper we are interested in the existence of complex-valued
harmonic morphisms from Lie groups equipped with left-invariant
Riemannian metrics. We give a new method for manufacturing solutions
on a large family of such spaces. Our following main result is a
wide-going generalization of the recent Theorem 12.1 of \cite{Gud-Sve-4}.

\begin{theorem}
Let $G=N\rtimes A$ be a semi-direct product of the connected and simply
connected Lie groups $A$ and $N$.  Let $G$ be equipped with a left-invariant
metric $g$ and $\g=\a\oplus\k\oplus\m$ be an orthogonal decomposition of the
Lie algebra $\g$ of $G$ such that $\a$ and $\n=\k\oplus\m$ are the Lie algebras
of $A$ and $N$, respectively.  Let $K$ be a closed simply connected subgroup
of $N$ with Lie algebra $\k$ such that
\begin{enumerate}
\item[(i)] $[\a,\k]\subset\k$,
\item[(ii)] $[\a,\m]\subset\m$,
\item[(iii)] $[\k\oplus\m,\k\oplus\m]\subset\k$,
\item[(iv)] $\text{\rm trace }\ad_Z=0$ for all $Z\in\m$,
\item[(v)] there exists a $\lambda\in\a^*$ such that for
  each $H\in\a$ and $Z\in\m$
$$(\ad_H+\ad_H^t)(Z)=2\lambda(H)\cdot Z.$$
\end{enumerate}
Then there exists a harmonic morphism $\Phi:G\to\rn^m$, where $m=\dim\m$.
\end{theorem}

We use our method to produce many new examples of complex-valued
harmonic morphisms from Lie groups.  We show how the construction
can be applied  to an important class of homogeneous
Hadamard manifolds including the well-known Carnot spaces,
see Theorem \ref{theo:Hadamard}.
We conclude this paper with a new recipe for constructing conformal
foliations on Lie groups producing harmonic morphisms, see
Theorem \ref{theo:foliations}.

For the general theory of harmonic morphisms, we refer to
the exhaustive book \cite{Bai-Woo-book} and the on-line
bibliography \cite{Gud-bib} of papers.

\section{Harmonic morphisms}

Let $M$ and $N$ be two manifolds of dimensions $m$ and $n$,
respectively. A Riemannian metric $g$ on $M$ gives rise to the
notion of a {\it Laplacian} on $(M,g)$ and real-valued {\it
harmonic functions} $f:(M,g)\to\rn$. This can be generalized to
the concept of {\it harmonic maps} $\phi:(M,g)\to (N,h)$ between
Riemannian manifolds, which are solutions to a semi-linear system
of partial differential equations, see \cite{Bai-Woo-book}.

\begin{definition}
  A map $\phi:(M,g)\to (N,h)$ between Riemannian manifolds is
  called a {\it harmonic morphism} if, for any harmonic function
  $f:U\to\rn$ defined on an open subset $U$ of $N$ with $\phi^{-1}(U)$
non-empty,
  $f\circ\phi:\phi^{-1}(U)\to\rn$ is a harmonic function.
\end{definition}

The following characterization of harmonic morphisms between
Riemannian manifolds is due to Fuglede and Ishihara.  For the
definition of horizontal (weak) conformality we refer to
\cite{Bai-Woo-book}.

\begin{theorem}\cite{Fug-1,Ish}
  A map $\phi:(M,g)\to (N,h)$ between Riemannian manifolds is a
  harmonic morphism if and only if it is a horizontally (weakly)
  conformal harmonic map.
\end{theorem}

The next result of Baird and Eells gives the theory of
harmonic morphisms a strong geometric flavour and shows that the
case when $n=2$ is particularly interesting. The conditions
characterizing harmonic morphisms are then independent of
conformal changes of the metric on the surface $N^2$.

\begin{theorem}\cite{Bai-Eel}\label{theo:B-E}
Let $\phi:(M^m,g)\to (N^n,h)$ be a horizontally (weakly) conformal
map between Riemannian manifolds. If
\begin{enumerate}
\item[(i)] $n=2$, then $\phi$ is harmonic if and only if $\phi$ has
minimal fibres at regular points;
\item[(ii)] $n\neq 2$, then two of the following conditions imply the other:
\begin{enumerate}
\item $\phi$ is a harmonic map, \item $\phi$ has minimal fibres at regular points,
\item $\phi$ is horizontally homothetic.
\end{enumerate}
\end{enumerate}
\end{theorem}

We are interested in complex-valued functions
$\phi,\psi:(M,g)\to\cn$ from Riemannian manifolds. In that
situation the metric $g$ induces the complex-valued Laplacian
$\tau(\phi)$ and the gradient $\text{grad}(\phi)$ with values in
the complexified tangent bundle $T^{\cn}M$ of $M$.  We extend the
metric $g$ to be complex bilinear on $T^{\cn} M$ and define the
symmetric bilinear operator $\kappa$ by
$$\kappa(\phi,\psi)= g(\text{grad}(\phi),\text{grad}(\psi)).$$ Two
maps $\phi,\psi: M\to\cn$ are said to be {\it orthogonal} if
$$\kappa(\phi,\psi)=0.$$  The harmonicity and horizontal
conformality of $\phi:(M,g)\to\cn$ are expressed by the relations
$$\tau(\phi)=0\ \ \text{and}\ \ \kappa(\phi,\phi)=0.$$

\begin{definition}\cite{Gud-8}
Let $(M,g)$ be a Riemannian manifold.  Then a set
$$\Omega=\{\phi_k:M\to\cn\ |\ k\in I\}$$ of
complex-valued functions is said to be an {\it orthogonal harmonic
family} on $M$ if, for all $\phi,\psi\in\Omega$,
$$\tau(\phi)=0\ \ \text{and}\ \ \kappa(\phi,\psi)=0.$$
\end{definition}

The next result shows that orthogonal harmonic families can be useful
for producing a variety of harmonic morphisms.

\begin{theorem}\cite{Gud-Sve-1}\label{theo:local-sol}
Let $(M,g)$ be a Riemannian manifold and
$$\Omega=\{\phi_k:M\to\cn\ |\ k=1,\dots ,n\}$$ be a finite orthogonal
harmonic family on $(M,g)$.  Let $\Phi:M\to\cn^n$ be the map given
by $\Phi=(\phi_1,\dots,\phi_n)$ and $U$ be an open subset of
$\cn^n$ containing the image $\Phi(M)$ of $\Phi$.
If $\H=\{h_i:U\to\cn\ |\ i\in\cn\}$
is a family of holomorphic functions then the family
$\F$ given by
$$\F=\{\psi:M\to\cn\ |\ \psi=h(\phi_1,\dots ,\phi_n),\ h\in\H\}$$
is an orthogonal harmonic family on $M$.  In particular, every element
of $\F$ is a harmonic morphism.
\end{theorem}

The problem of finding an orthogonal harmonic family on a Riemannian
manifold can often be reduced to finding a harmonic morphism with values
in $\rn^n$.

\begin{proposition}\cite{Gud-Sve-4}
Let $\Phi:(M,g)\to \rn^n$ be a harmonic morphism
from a Riemannian manifold to the standard Euclidean
$\rn^n$ with $n\ge 2$.  If $V$ is an isotropic subspace of $\cn^n$ then
$$\Omega_V=\{\phi_v(x)=(\Phi(x),v)|\ v\in V\}$$ is an orthogonal harmonic
family of complex-valued functions on $(M,g)$.
\end{proposition}

Here $(\cdot,\cdot )$ refers to the standard symmetric bilinear form
on the complex linear space $\cn^n$.

\section{Harmonic morphisms from Lie groups}

In this section we introduce a new method for manufacturing complex-valued
harmonic morphisms from a large class of Lie groups.  This is a wide-going
generalization of a construction recently presented in \cite{Gud-Sve-4}.

\begin{theorem}\label{theo:morphisms}
Let $G=N\rtimes A$ be a semi-direct product of the connected and simply
connected Lie groups $A$ and $N$.  Let $G$ be equipped with a left-invariant
metric $g$ and $\g=\a\oplus\k\oplus\m$ be an orthogonal decomposition of the
Lie algebra $\g$ of $G$ such that $\a$ and $\n=\k\oplus\m$ are the Lie algebras
of $A$ and $N$, respectively.  Let $K$ be a closed simply connected subgroup
of $N$ with Lie algebra $\k$ such that
\begin{enumerate}
\item[(i)] $[\a,\k]\subset\k$,
\item[(ii)] $[\a,\m]\subset\m$,
\item[(iii)] $[\k\oplus\m,\k\oplus\m]\subset\k$,
\item[(iv)] $\text{\rm trace }\ad_Z=0$ for all $Z\in\m$,
\item[(v)] there exists a $\lambda\in\a^*$ such that for
  each $H\in\a$ and $Z\in\m$
$$(\ad_H+\ad_H^t)(Z)=2\lambda(H)\cdot Z.$$
\end{enumerate}
Then there exists a harmonic morphism $\Phi:G\to\rn^m$, where $m=\dim\m$.
\end{theorem}

The following proof of Theorem \ref{theo:morphisms} is an
amended version of the proof of the special case presented in
Theorem 12.1 of \cite{Gud-Sve-4}.

\begin{proof}
The subalgebra $\k$ is an ideal of both $\g$ and $\n$, so $K$ is a
normal subgroup of $G$ and $N$.  The group $N/K$
is connected and Abelian since $[\n,\n]\subset\k$.
Both $N$ and $K$ are simply connected so $N/K$ is simply connected
and hence isomorphic to $\rn^{m}$.

Equip $N$ with the metric induced by $g$ on $G$ and suppose that $N/K$ has the
unique left-invariant metric $h$ such that the homogeneous projection
$$\pi:(N,g)\to (N/K,h)$$ with $\pi(n)=n K$ is a Riemannian submersion.
Define the map $\Psi:G\to N$ by $\Psi:na\mapsto n$ and let
$\Phi:G\to N/K$ be the composition $\Phi=\pi\circ\Psi$.
We will show that $\Phi$ is a harmonic morphism.

For a point $nK\in N/K$ we have $\Phi^{-1}(nK)=nKA=nAK$.
The tangent space of $G$ at $e$ is the Lie algebra $\g$ and
the vertical and horizontal spaces of $\Phi$ at $e\in G$ are given by
$$\V_{e}=\a+\k\ \ \text{and}\ \ \H_{e}=\m.$$
At a generic point $na\in G$ the vertical and horizontal spaces
of $\Phi$ are given by the left translate by $dL_{na}$ i.e.
$$\V_{na}=(dL_{na})_{e}(\a+\k)\ \
\text{and}\ \ \H_{na}=(dL_{na})_{e}(\m).$$
\begin{eqnarray*}
d\Phi_{na}((dL_{na})_{e}(Z))
&=&\frac d{dt}\Big[\Phi(L_{na}(\exp(tZ)))\Big]_{t=0}\\
&=&\frac d{dt}\Big[\Phi(n\exp(t\sigma(a)(Z)a))\Big]_{t=0}\\
&=&\frac d{dt}\Big[(nK)(\exp(t\sigma(a))(Z)K)]_{t=0}\\
&=&(dL_{nK})_{eK}(\sigma(a)(Z)),
\end{eqnarray*}
where $\sigma:A\to\text{Aut}(\m)$ is given by $\sigma(a)(Z)=\Ad_a(Z).$
This implies that if $$\rho=d\sigma:\a\to\text{End}(\m)$$ then $\rho(H)$ acts
on $\m$ by the adjoint representation.  According to condition (v)
this is given by
$$\rho(H)(Z)=[H,Z]=\lambda(H)Z+\frac 12(\ad_H-\ad_H^t)Z.$$
The map $\rho$ takes values in $\rn\cdot\text{I}\oplus\so\m$ which is the
Lie algebra of the conformal group on $\m$.  Hence for each
$a\in A$ the map $\sigma(a):\m\to\m$ is conformal.
This implies that for $Z,W\in\m$ we get
\begin{eqnarray*}
& &h_{nK}((dL_{nK})_{eK}(\sigma(a)(Z)),
                (dL_{nK})_{eK}(\sigma(a)(W)))\\
&=&h_{eK}(\sigma(a)(Z),\sigma(a)(W))\\
&=&\mu(a)^2h_{eK}(Z,W),
\end{eqnarray*}
where $\mu:A\to\rn^+$ is the dilation of $\sigma(a)$.  This is clearly
constant along horizontal curves with respect to $\Phi$.  This means that $\Phi$
is horizontally homothetic so for proving that $\Phi$ is a harmonic morphism
it is sufficient to show that it has minimal fibres.

Since $\Phi^{-1}(nK)=nKA$ and $L_{n}$ is an isometry on $G$ it
is enough to show that the fibre $\Phi^{-1}(eK)$ is minimal.
Let $\{H_j\}$ be an orthonormal basis for
$\a$ and $\{X_{i}\}$ an orthonormal basis for $\k$. Let $\mu^\V$ be
the mean curvature vector field along $\Phi^{-1}(eK)$.
For a generic element $Z\in\m$ we have $[H_j,Z]\in \m$ and then
conditions (i), (iii) and (v) give
\begin{eqnarray*}
\ip {\mu^{\V}}Z&=&\sum_{j}{\ip{\nab{H_j}{H_j}}Z}+{\sum_{i}{\ip{\nab{X_{i}}{X_{i}}}Z}}\\
&=&\sum_{j}{\left<[Z,H_j],H_j\right>}+\sum_{i}{\left<[Z,X_i],X_i\right>}\\
&=&\text{trace }\ad_{Z}\\
&=&0.
\end{eqnarray*}
These calculations show that the fibres are minimal and
hence $\Phi$ is a harmonic morphism.
\end{proof}

In order to explain condition (v) in Theorem \ref{theo:morphisms}
we state the following result.

\begin{lemma}\label{lemm:trivial}
Let $V$ be a real vector space equipped with a Euclidean scalar product
$\ip\cdot\cdot$.
Let $L:V\to V$ be a linear operator on $V$.  Then the following conditions
are equivalent.
\begin{enumerate}
\item[(i)] there exist a $\lambda\in\rn$ such that $L=\lambda\cdot I+(L-L^t)/2$,
\item[(ii)] if $Z,W\in V$ such that $|Z|=|W|$ then $\ip {LZ}Z=\ip {LW}W$,
\item[(iii)] if $Z,W\in V$ such that $|Z|=|W|=1$ and $\ip ZW=0$ then
$$\ip{LZ}Z-\ip {LW}W=0\ \ \text{and}\ \ \ip{LZ}W-\ip {LW}Z=0.$$
\end{enumerate}
\end{lemma}

\begin{proof}
The proof is an easy exercise left to the reader.
\end{proof}

\begin{remark}
If the Lie algebra $\k$ is semi-simple then $[\k,\k]=\k$.  This means that
independent of how we pick $\a$ and $\m$ satisfying the conditions in
Theorem \ref{theo:morphisms} the resulting Lie algebra $\g=\a\oplus\k\oplus\m$
will not be solvable.  This gives a wide class of Lie algebras not
covered by Theorem 12.1 of \cite{Gud-Sve-4}.
\end{remark}

\section{The case of $\dim(\a,\k,\m)=(2,0,2)$}

Let $\g$ be a $4$-dimensional Lie algebra equipped with a Euclidean
metric and $\g=\a\oplus\k\oplus\m$ be an orthogonal decomposition
of $\g$ satisfying the conditions of Theorem \ref{theo:morphisms}.
If $\text{dim}(\a,\k,\m)=(2,0,2)$ then there exists an orthonormal
basis $\{A,B,Z,W\}$ for $\g$ such that $A,B\in\a$, $Z,W\in\m$ and
$$\lb AB=aA+bB,$$
$$\lb AZ=\alpha Z+\beta W,\ \ \lb AW=-\beta Z+\alpha W,$$
$$\lb BZ= x Z+y W,\ \ \lb BW=-y Z+x W,$$
The Jacobi identity gives the following conditions on the real coefficients
$$a\alpha+bx=0=a\beta+by,$$
The Lie algebra $\g$ is solvable, the vertical foliation generated by $\a$
is totally geodesic and the horizontal distribution $\m$ is integrable.

\section{The case of $\dim(\a,\k,\m)=(1,1,2)$}

Let $\g$ be a $4$-dimensional Lie algebra equipped with a Euclidean
metric and $\g=\a\oplus\k\oplus\m$ be an orthogonal decomposition
of $\g$ satisfying the conditions of Theorem \ref{theo:morphisms}.
If $\text{dim}(\a,\k,\m)=(1,1,2)$ then there exists an orthonormal basis
$\{A,X,Z,W\}$ for $\g$ such that $A\in\a$, $X\in\k$, $Z,W\in\m$ and
$$\lb AX=\lambda X,$$
$$\lb AZ=\alpha Z+\beta W,\ \ \lb AW=-\beta Z+\alpha W,$$
$$\lb ZW=\theta X.$$
The Jacobi identity gives the following condition on the real
coefficients $$\theta(\lambda-2\alpha)=0.$$
The Lie algebra $\g$ is solvable and the horizontal distribution $\m$
is integrable if and only if $\theta=0$.

\section{The case of $\dim(\a,\k,\m)=(0,2,2)$}

Let $\g$ be a $4$-dimensional Lie algebra equipped with a Euclidean
metric and $\g=\a\oplus\k\oplus\m$ be an orthogonal decomposition
of $\g$ satisfying the conditions of Theorem \ref{theo:morphisms}.
If $\text{dim}(\a,\k,\m)=(0,2,2)$ then there exists an orthonormal basis
$\{X,Y,Z,W\}$ for $\g$ such that $X,Y\in\k$, $Z,W\in\k$ and
$$\lb XY=zX+wY,$$
$$\lb ZX=r X+s Y,\ \ \lb ZY=tX-rY,$$
$$\lb WX=\rho X+\sigma Y,\ \ \lb WY=\tau X-\rho Y,$$
$$\lb ZW=\theta X.$$
The Jacobi identity gives 8 quadratic equations which
can be solved in order to construct a variety of solutions.
The horizontal distribution $\m$ is integrable if and only if $\theta=0$.
The fibres are totally geodesic if and only if $r=s+t=\rho=\sigma+\tau=0$.
We present our examples by simply listing the non-vanishing Lie brackets
in each case.

\begin{example}
$$\lb XY=zX+wY,$$
$$\lb ZX=-\frac{wt}zX-\frac{w^2t}{z^2}Y,\ \ \lb ZY=tX+\frac{wt}zY.$$
$$\lb WX=-\frac{w\tau}zX-\frac{w^2\tau}{z^2}Y,\ \ \lb WY=\tau X+\frac{w\tau}zY.$$
\end{example}

\begin{example}
$$\lb XY=wY,$$
$$\lb ZX=sY,\ \ \lb WX=\sigma Y.$$
\end{example}

\begin{example}
$$\lb ZX=r X+s Y,\ \ \lb ZY=tX-rY,$$
$$\lb WX=\rho X+\frac{s\rho}r Y,\ \ \lb WY=\frac{t\rho}r X-\rho Y,$$
$$\lb ZW=\theta X.$$
\end{example}

\begin{example}
$$\lb ZX=s Y,\ \ \lb ZY=tX,$$
$$\lb WX=\sigma Y,\ \ \lb WY=\frac{t\sigma}s X,$$
$$\lb ZW=\theta X.$$
\end{example}

\begin{example}
$$\lb WX=\rho X+\sigma Y,\ \ \lb WY=\tau X-\rho Y,$$
$$\lb ZW=\theta X.$$
\end{example}

\begin{example}
$$\lb ZY=tX,\ \ \lb WY=\tau X,$$
$$\lb ZW=\theta X.$$
\end{example}

\section{The case of $\dim(\a,\k,\m)=(2,1,2)$}

Let $\g$ be a $5$-dimensional Lie algebra equipped with a Euclidean
metric and $\g=\a\oplus\k\oplus\m$ be an orthogonal decomposition
of $\g$ satisfying the conditions of Theorem \ref{theo:morphisms}.
If $\text{dim}(\a,\k,\m)=(2,1,2)$ then there exists an orthonormal
basis $\{A,B,X,Z,W\}$ for $\g$ such that $A,B\in\a$,
$X\in\k$, $Z,W\in\m$ and
$$\lb AB=aA+bB,$$
$$\lb AX=\lambda X,\ \ \lb BX=\mu X,$$
$$\lb AZ=\alpha Z+\beta W,\ \ \lb AW=-\beta Z+\alpha W,$$
$$\lb BZ= x Z+y W,\ \ \lb BW=-y Z+x W,$$
$$\lb ZW=\theta X.$$
The Jacobi identity gives 5 quadratic conditions for the real coefficients.
$$\theta(\lambda-2\alpha)=0=\theta(\mu-2x),$$
$$a\alpha+bx=0=a\beta+by,$$
$$a\lambda+b\mu=0.$$
By solving these equations, we get the following examples.

\begin{example}
$$\lb AB=aA+bB,$$
$$\lb AX=-\frac{b\mu}a X,\ \ \lb BX=\mu X,$$
$$\lb AZ=-\frac{bx}a Z-\frac{by}a W,\ \ \lb AW=\frac{by}a Z-\frac{bx}a W,$$
$$\lb BZ= x Z+y W,\ \ \lb BW=-y Z+x W.$$
\end{example}

\begin{example}
$$\lb AB=aA+bB,$$
$$\lb AX=-2\frac{bx}a X,\ \ \lb BX=2x X,$$
$$\lb AZ=-\frac{bx}a Z-\frac{by}a W,\ \ \lb AW=\frac{by}a Z-\frac{bx}a W,$$
$$\lb BZ= x Z+y W,\ \ \lb BW=-y Z+x W,$$
$$\lb ZW=\theta X.$$
\end{example}

\begin{example}
$$\lb AX=2\alpha X,\ \ \lb BX=2x X,$$
$$\lb AZ=\alpha Z+\beta W,\ \ \lb AW=-\beta Z+\alpha W,$$
$$\lb BZ= x Z+y W,\ \ \lb BW=-y Z+x W,$$
$$\lb ZW=\theta X.$$
\end{example}

\begin{example}
$$\lb AB=bB,$$
$$\lb AX=2\alpha X,$$
$$\lb AZ=\alpha Z+\beta W,\ \ \lb AW=-\beta Z+\alpha W,$$
$$\lb ZW=\theta X.$$
\end{example}

\begin{example}
$$\lb AX=\lambda X,\ \ \lb BX=\mu X,$$
$$\lb AZ=\alpha Z+\beta W,\ \ \lb AW=-\beta Z+\alpha W,$$
$$\lb BZ= x Z+y W,\ \ \lb BW=-y Z+x W.$$
\end{example}

\begin{example}
$$\lb AB=bB,$$
$$\lb AX=\lambda X,$$
$$\lb AZ=\alpha Z+\beta W,\ \ \lb AW=-\beta Z+\alpha W,$$
\end{example}

\section{The case of $\dim(\a,\k,\m)=(1,2,2)$}

Let $\g$ be a $5$-dimensional Lie algebra equipped with a Euclidean
metric and $\g=\a\oplus\k\oplus\m$ be an orthogonal decomposition
of $\g$ satisfying the conditions of Theorem \ref{theo:morphisms}.
If $\text{dim}(\a,\k,\m)=(1,2,2)$ then there exists an orthonormal
basis $\{A,X,Y,Z,W\}$ for $\g$ such that $A\in\a$,
$X,Y\in\k$ and $Z,W\in\m$ and
$$\lb AX=\gamma X+\delta Y,\ \ \lb AY=cX+dY,$$
$$\lb AZ=\alpha Z+\beta W,\ \ \lb AW=-\beta Z+\alpha W,$$
$$\lb XY=zX+wY,$$
$$\lb ZX=rX+sY,\ \ \lb ZY=tX-rY,$$
$$\lb WX=\rho X+\sigma Y,\ \ \lb WY=\tau X-\rho Y,$$
$$\lb ZW=\theta X.$$
The Jacobi identity gives 20 quadratic equations which
can be solved in order to construct a large variety of solutions.
In this section we focus our attention on cases with non-integrable
horizontal distribution ($\theta\neq 0$) and non-totally geodesic fibres
($r\neq 0$ or $s+t\neq 0$).

\begin{example}
$$\lb AY=dY,$$
$$\lb ZX=rX,\ \ \lb ZY=-rY,$$
$$\lb WX=\rho X,\ \ \lb WY=-\rho Y,$$
$$\lb ZW=\theta X.$$
\end{example}

\begin{example}
$$\lb AX=2\alpha X,\ \ \lb AY=\alpha X+3\alpha Y,$$
$$\lb AZ=\alpha Z,\ \ \lb AW=\alpha W,$$
$$\lb ZX=rX+rY,\ \ \lb ZY=-rX-rY,$$
$$\lb WX=\rho X+\rho Y,\ \ \lb WY=-\rho X-\rho Y,$$
$$\lb ZW=\theta X.$$
\end{example}

\begin{example}
$$\lb AX=2\alpha X,\ \ \lb AY=cX+\alpha Y,$$
$$\lb AZ=\alpha Z,\ \ \lb AW=\alpha W,$$
$$\lb ZY=tX,\ \ \lb WY=\tau X,$$
$$\lb ZW=\theta X.$$
\end{example}

\begin{example}
$$\lb AX=2\alpha X,\ \ \lb AY=3\alpha Y,$$
$$\lb AZ=\alpha Z,\ \ \lb AW=\alpha W,$$
$$\lb ZX=sY,$$
$$\lb WX=\sigma Y,$$
$$\lb ZW=\theta X.$$
\end{example}

\begin{example}
$$\lb AX=2\alpha X,\ \ \lb AY=\frac{\alpha r}sX+3\alpha Y,$$
$$\lb AZ=\alpha Z,\ \ \lb AW=\alpha W,$$
$$\lb ZX=rX+sY,\ \ \lb ZY=-\frac{r^2}sX-rY,$$
$$\lb ZW=\theta X.$$
\end{example}

\section{The case of $\dim(\a,\k,\m)=(1,n,2)$}

\begin{example}
Let $\g$ be an $(n+3)$-dimensional Lie algebra equipped with a Euclidean
metric and $\g=\a\oplus\k\oplus\m$ be an orthogonal decomposition
of $\g$.  Let $\{A,X_1,\dots ,X_n,Z,W\}$ be an orthonormal basis
for $\g$ such that $A\in\a$, $X_1,\dots,X_n\in\k$, $Z,W\in\m$ and
\begin{eqnarray*}
\lb A{X_k} &=&c_{k1}X_1+\dots +c_{kn}X_n\\
\lb AZ &=& \alpha Z+\beta W\\
\lb AW &=&-\beta  Z+\alpha W.
\end{eqnarray*}
The Jacobi identity is satisfied for any choice of real numbers
$\alpha$, $\beta$ and $c_{kj}$. The Lie algebra $\g$ is solvable and
satisfies the conditions in Theorem \ref{theo:morphisms}.

If $n\geq 2$, $c_{kj}=0$ for $k<j$ and for some $r<s$ we have
$c_{rr}=c_{ss}$ and $c_{sr}\neq 0$ or if $\beta\neq 0$, then
$\ad_{A}$ is not diagonalizable and we get new examples of
harmonic morphisms not covered by Theorem 12.1 of \cite{Gud-Sve-4}.
\end{example}

\begin{example}
Let $\g$ be an $(n+3)$-dimensional Lie algebra equipped with a Euclidean
metric and $\g=\a\oplus\k\oplus\m$ be an orthogonal decomposition
of $\g$.  Let $\{A,X_1,\dots ,X_n,Z,W\}$ be an orthonormal basis
for $\g$ such that $A\in\a$, $X_1,\dots,X_n\in\k$ and $Z,W\in\m$ and
\begin{eqnarray*}
\lb A{X_k} &=&X_k\\
\lb AZ &=& \alpha Z+\beta W\\
\lb AW &=&-\beta  Z+\alpha W\\
\lb ZW &=& X_1.
\end{eqnarray*}
Then the Jacobi identity implies that $\alpha=1/2$.
The Lie algebra $\g$ is solvable and satisfies the conditions in
Theorem \ref{theo:morphisms}.
The horizontal distribution $\m$ is not integrable.
\end{example}

\section{Homogeneous Hadamard manifolds}

Homogeneous Hadamard manifolds are simply connected Riemannian
homogeneous spaces with non-positive sectional curvature.
Every such manifold $(M^m,g)$ is diffeomorphic to the standard $\rn^m$
and it is actually isometric to a solvable Riemannian Lie group,
see \cite{Wol}, \cite{Hei} and \cite{Aze-Wil-1}.  Important examples are the
irreducible Riemannian symmetric spaces of non-compact type
and the Damek-Ricci spaces, see \cite{Ber-Tri-Van}.

If $S$ is a Riemannian Lie group with Lie algebra $\s$, of non-positive
curvature, then there exists an orthogonal decomposition $\s=\a\oplus\n$ of $\s$,
where $\n=[\s,\s]$.  The subalgebra $\n$ is nilpotent and $\a$ is Abelian.
It follows from the Jacobi identity that if $H_1,H_2\in\a$ and $X\in\n$ then
$$(\ad_{H_1}\circ\ad_{H_2}-\ad_{H_2}\circ\ad_{H_1})(X)=0$$
i.e. the adjoint action of $\a$ on $\n$ is Abelian.
This implies that the complexification  $\n^\cn$ of $\n$ is a direct sum
$$\n^\cn=\bigoplus_\lambda\n^\cn_\lambda$$ of the non-trivial root spaces
$$\n^\cn_\lambda=\{X\in\n^\cn|\ (\ad_H-\lambda(H))^kX=0
\ \ \text{for some $k>0$ and all $H\in\a$}\},$$
where $\lambda\in(\a^\cn)^*$ are the corresponding roots.
For a root $\lambda=\alpha\pm i\beta\in(\a^\cn)^*$ the
generalized root space $\n_{\alpha,\beta}$ is given by
$$\n_{\alpha,\beta}=\n\cap(\n^\cn_{\lambda}\oplus\n^\cn_{\bar\lambda}).$$
Then the vector space $\n$ is the direct sum of generalized root spaces
$$\n=\bigoplus_{\lambda=\alpha\pm i\beta}\n_{\alpha,\beta}.$$

\begin{definition}
Let $V$ be a complex vector space equipped with a Hermitian scalar
product and $L:V\to V$ be a linear operator on $V$.  Further define
$N(L):V\to V$ by $$4N(L)=(L+L^*)^2+[L+L^*,L-L^*].$$
Then the operator $L$ is said to be
\begin {enumerate}
\item[(i)] {\it normal} if $L L^*=L^* L$, and
\item[(ii)] {\it almost normal} if $N(L)$ is semi-positive definite
i.e. for all $Z\in V$ $$\ip {N(L)(Z)}Z\ge 0.$$
\end{enumerate}
\end{definition}

For homogeneous Hadamard manifolds, Azencott and Wilson have in
\cite{Aze-Wil-1} shown that the complex linear extension
 $\ad_H|_{\n^\cn}:\n^\cn\to\n^\cn$ of $\ad_H|_\n$
is almost normal for all $H\in\a$.  It is easily seen that any normal
operator $L$ is also almost normal.  We shall now focus our attention on
the special case of normality.

\begin{lemma}\label{lemm:normal-root}
Let $\n_{\alpha,\beta}$ be a generalized root space of $\n$ such
that the operator $\ad_{H}|_{\n_{\alpha,\beta}}$
is normal for all $H\in\a$. Then
$$\ad_{H}|_{\n_{\alpha,\beta}}=\alpha(H)I_{\n_{\alpha,\beta}}
+\frac 12(\ad_{H}-\ad_{H}^*)|_{\n_{\alpha,\beta}}$$
\end{lemma}

\begin{proof}
The complex linear extension
$$\ad_H|_{\n^\cn_\lambda\oplus\n^\cn_{\bar\lambda}}
:\n^\cn_\lambda\oplus\n^\cn_{\bar\lambda}
\to\n^\cn_\lambda\oplus\n^\cn_{\bar\lambda}$$
of $\ad_H|_{\n_{\alpha,\beta}}$ is normal and hence diagonalizable over $\cn$.  Let
$\{X_{1},\ldots,X_{n}\}$ and $\{Y_{1},\ldots,Y_{n}\}$ be orthonormal
basis of eigenvectors for $\n^\cn_{\lambda}$ and
$\n^\cn_{\bar\lambda}$, respectively.  For $X\in\n_{\lambda}$ we then get
\begin{eqnarray*}
\ad_{H}^{*}(X)&=&\sum_{k=1}^{n}(\left<\ad_{H}^{*}(X),X_{k}\right>X_k
+\left<\ad_{H}^{*}(X),Y_{k}\right>Y_{k})\\
&=&\sum_{k=1}^{n}(\left<X,\ad_{H}(X_{k})\right>X_{k}
+\left<X,\ad_{H}(Y_{k})\right>Y_{k})\\
&=&\sum_{k=1}^{n}\left<X,\lambda(H)X_{k}\right>X_{k}\\
&=&\sum_{k=1}^{n}{\bar\lambda(H)}\left<X,X_{k}\right>X_{k}\\
&=&{\bar\lambda(H)}X.
\end{eqnarray*}
Similarily, we see that if $Y\in\n_{\bar{\lambda}}$ then
$\ad_{H}^{*}(Y)=\lambda(H)Y$ and hence
$$(\ad_H+\ad^*_H)(Z)=(\lambda(H)+\bar\lambda(H))Z=2\alpha(H)Z$$
for all $Z\in\n_{\lambda}\oplus\n_{\bar{\lambda}}$.
\end{proof}

We shall now apply our general result of
Theorem \ref{theo:morphisms} to construct harmonic morphisms
from an important class of homogeneous Hadamard manifolds.

\begin{theorem}\label{theo:Hadamard}
Let the solvable Riemannian  Lie group $S$ be a homogeneous Hadamard
manifold with Lie algebra $\s=\a+\n$.  Furthermore assume that there exists
a generalized root space $\n_{\alpha,\beta}$ of $\n$ of dimension
$m\geq 2$ such that
\begin{enumerate}
\item[(i)] $\ad_{H}|_{\n_{\alpha,\beta}}$ is normal for all $H\in\a$,
\item[(ii)] $[\n,\n]$ is contained in the orthogonal complement
 $\n_{\alpha,\beta}^\perp$ of $\n_{\alpha,\beta}$ in $\n$,
\item[(iii)] $\ad_{H}(\n_{\alpha,\beta}^{\perp})
\subset\n_{\alpha,\beta}^{\perp}$ for all $H\in\a$.
\end{enumerate}
Then there exists a harmonic morphism $\Phi:S\to\rn^m$.
\end{theorem}

\begin{proof}
In this situation we have an orthogonal decomposition
$\s=\a\oplus\k\oplus\m$ of the Lie algebra $\s$ where
$\m=\n_{\alpha,\beta}$ and $\k=\n_{\alpha,\beta}^\perp$.
It is now easy to check, using Lemma \ref{lemm:normal-root}
and the fact that $\n$ is nilpotent, that the conditions
of Theorem \ref{theo:morphisms} are satisfied.
\end{proof}

\begin{example}
The Carnot spaces form an interesting family of
homogeneous Hadamard manifolds, see \cite{Pan}.
Such a space $(S,g)$ is a solvable Riemannian Lie group
with Lie algebra $\s$ and an orthogonal decomposition
$$\s=\a\oplus\bigoplus_{r=1}^k\n_r$$
of $\s$ such that $\a$ is a one dimensional subalgebra and the
adjoint action of $\a$ satisfies
$$\ad_H(X_r)=[H,X_r]=r\cdot X_r,$$
for all $H\in\a$ and $X_r\in\n_r$.  It is an immediate consequence
of the Jacobi identity that
$$[\n_r,\n_s]\subset\n_{r+s},$$
where $\n_{r+s}=0$ if $k\le r+s$.

Let us assume that the dimension of $\n_1$ is at least $2$
and put $\m=\n_1$ and $\k=\n_2\oplus\cdots\oplus\n_k$.   Then
we have an orthogonal decomposition $\s=\a\oplus\k\oplus\m$
satisfying the conditions of Theorem \ref{theo:morphisms}.
\end{example}

\section{Examples}

We have constructed many new examples of
complex-valued harmonic morphisms from Riemannian Lie groups.
It turns out that several of these spaces are actually homogeneous
Hadamard manifolds.  This can be checked by calculating their
sectional curvatures.  For this we have used a Maple programme
written by Valentine Svensson, see \cite{Val}.
Here we list some of these examples and conditions for the real
parameters, ensuring non-positive curvature in each case.
\vskip .4cm

\noindent
Example 7.1: $a^3<b^2\mu$,
             $a^3<b^2x$,
             $b^2<a\mu$,
             $b^2<ax$ and
             $0<a,\mu,x$.
\vskip .5cm

\noindent
Example 7.2: $\theta^2a^2<8x^2(a^2+b^2)$,
             $a^3<b^2x$, $b^2<ax$ and
             $0<a,x$.
\vskip .5cm

\noindent
Example 7.3: $\theta^2<8(x^2+a^2)$.
\vskip .5cm

\noindent
Example 7.4: $0<\alpha b$ and
             $\theta^2<8\alpha^2$.
\vskip .5cm

\noindent
Example 7.5: $0<\mu x+\lambda\alpha$.
\vskip .5cm

\noindent
Example 7.6: $0<\alpha b$,
             $0<\alpha\lambda$ and
             $0<b\lambda$.
\vskip .5cm

\noindent
Example 8.2: $4(\rho^2+r^2)<23\alpha^2$,
             $\theta^2<8\alpha^2+4r^2$ and
             $\theta^2<8\alpha^2+4\rho^2$.
\vskip .5cm

\noindent
Example 8.3: $c^2<16\alpha^2$,
             $\theta^2+\tau^2<8\alpha^2$,
             $\theta^2+t^2<8\alpha^2´$ and
             $t^2+\tau^2+c^2<8\alpha^2$.
\vskip .5cm

\noindent
Example 8.4: $s^2<12\alpha^2$,
             $\sigma^2<12\alpha^2$,
             $\theta^2<3s^2+8\alpha^2$ and
             $\theta^2<3\sigma^2+8\alpha^2$.
\vskip .5cm

\noindent
Example 8.5: $r^2<16s^2$,
             $\theta^2<8\alpha^2$,
             $2r^2s^2+r^4+\alpha^2r^2+s^4<24\alpha^2s^2$,
             $\theta^2s^2+r^4<3s^4+2r^2s^2+8\alpha^2s^2$ and
             $s^4<12\alpha^2s^2+2r^2s^2+3r^4$.
\vskip .5cm

\noindent
Example 9.2: no conditions are needed.
\vskip .5cm

\section{Conformal foliations on Lie groups}

In this section we present a new method for constructing
left-invariant foliations on a wide class of Lie groups producing
harmonic morphisms.  For the theory of foliations producing
harmonic morphisms we recommend the standard reference
\cite{Bai-Woo-book}.

\begin{theorem}\label{theo:foliations}
Let $G$ be a Lie group equipped with a left-invariant Riemannian metric.
Let $\g=\a\oplus\k\oplus\m$ be an orthogonal decomposition of the Lie
algebra $\g$ of $G$ such that
\begin{enumerate}
\item[(i)] $\a\oplus\k$ is a subalgebra of $\g$,
\item[(ii)] $[\a,\m]\subset\k\oplus\m$,
\item[(iii)] $[\k\oplus\m,\k\oplus\m]\subset\k$.
\end{enumerate}
Furthermore we assume that if $H\in\a$ and $Z,W\in\m$ such that $|Z|=|W|=1$
and $\ip ZW=0$, then
\begin{enumerate}
\item[(iv)] $\ip{\ad_HZ}Z-\ip {\ad_HW}W=0$,
\item[(v)] $\ip{\ad_HZ}W+\ip {\ad_HW}Z=0$,
\item[(vi)] $\text{\rm trace }\ad_Z=0$.
\end{enumerate}
Under the above conditions the distribution $\V=\a\oplus\k$
is integrable and the corresponding foliation produces
harmonic morphisms.
\end{theorem}

\begin{proof}
For the integrability of the distribution $\V$ is trivial.
The minimality of the leaves of the foliation $\V$,
is an immediate consequence of the following calculations,
implied by conditions (ii), (iii) and (vi).
\begin{eqnarray*}
\text{trace}B^\V&=&\sum_k\bigl(\sum_r(\ip{[Z_k,A_r]}{A_r})
+\sum_s(\ip{[Z_k,X_s]}{X_s})\bigr)Z_k\\
&=&\sum_k(\text{trace }\ad_{Z_k}) Z_k\\
&=&0.
\end{eqnarray*}

For the conformality of the foliation $\V$, we see that
conditions (iii), (iv) and (v) obtain
\begin{eqnarray*}
B^\H(Z,Z)-B^\H(W,W)
&=&\sum_r(\ip{[A_r,Z]}Z-\ip{[A_r,W]}W)A_r\\
&=&0
\end{eqnarray*}
and
\begin{eqnarray*}
B^\H(Z,W)
&=&\sum_r(\ip{[A_r,Z]}Y+\ip{[A_r,W]}X)A_r\\
&=&0.
\end{eqnarray*}

It follows from Corollary 4.7.7 of \cite{Bai-Woo-book} that if the dimension
of $\m$ is $2$ then we are done.  If that is not the case let $\omega$ be
the real-valued $1$-form on $G$ given by
$$\omega(E)=\frac{n-2}n\sum_k\ip{\V\nab{Z_k}{Z_k}}E
-\sum_r\ip{\H\nab{A_r}{A_r}}E-\sum_s\ip{\H\nab{X_s}{X_s}}E.$$
To prove the statement in this case, it is sufficient to show that $d\omega=0$
i.e. the $1$-form $\omega$ is closed.  Let $A\in\a$, $X\in\k$ and $Z\in\m$ then
\begin{eqnarray*}
\omega(A)&=&\frac{n-2}n\sum_k\ip{\V\nab{Z_k}{Z_k}}A\\
&=&\frac{n-2}n\sum_k\ip{[A,Z_k]}{Z_k}.
\end{eqnarray*}
Note that since the vector fields that we have chosen are left-invariant
the function $\omega(A)$ is constant.  Next we see that
\begin{eqnarray*}
\omega(X)&=&\frac{n-2}n\sum_k\ip{\V\nab{Z_k}{Z_k}}X\\
&=&\frac{n-2}n\sum_k\ip{[X,Z_k]}{Z_k}\\
&=&0
\end{eqnarray*}
and similarly
\begin{eqnarray*}
\omega(Z)&=&-\sum_r\ip{\H\nab{A_r}{A_r}}Z-\sum_s\ip{\H\nab{X_s}{X_s}}Z\\
&=&-\sum_r\ip{[Z,A_r]}{A_r}-\sum_s\ip{[Z,X_s]}{X_s}\\
&=&-\text{trace }\ad_Z\\
&=&0.
\end{eqnarray*}
If $E,F\in\g$ are left-invariant then the exterior derivative
$d\omega$ satisfies
$$d\omega(E,F)=E(\omega(F))-F(\omega(E))-\omega([E,F])=-\omega([E,F])=0,$$
since $\omega$ vanishes on $[\g,\g]\subset\k\oplus\m$.
\end{proof}

\section{Acknowledgements}
The authors are grateful to Martin Svensson
for very useful discussions on this work.

\end{document}